\documentclass[12pt, article]{amsart}
\usepackage{amsmath, amsthm, amscd, amsfonts, amssymb, graphicx, color}
\usepackage[bookmarksnumbered, colorlinks, plainpages]{hyperref}
\usepackage{color}

\newtheorem{theorem}{Theorem}[section]

\newtheorem{corollary}[theorem]{Corollary}
\theoremstyle{definition}

\theoremstyle{remark}
\newtheorem{remark}[theorem]{Remark}
\numberwithin{equation}{section}

\begin{document}

\title{An extension of the P\'{o}lya--Szeg\"{o} operator inequality}

\author[D.T. Hoa, M.S. Moslehian and C. Conde, P. Zhang] {Dinh Trung Hoa$^1$, Mohammad Sal Moslehian$^2$, Cristian Conde$^3$ and Pingping Zhang$^4$}
\address{$^1$Institute of Research and Development, Duy Tân University, Da Nang, Viet Nam}
\email{trunghoa.math@gmail.com}

\address{$^2$Department of Pure Mathematics, Center of Excellence in
Analysis on Algebraic Structures (CEAAS), Ferdowsi University of
Mashhad, P. O. Box 1159, Mashhad 91775, Iran}
\email{moslehian@um.ac.ir, moslehian@member.ams.org}

\address{$^3$Instituto de Ciencias, Universidad Nacional de Gral. Sarmiento, J.	M. Gutierrez 1150, (B1613GSX) Los Polvorines and
	Instituto Argentino de Matem\'atica ``Alberto P. Calder\'on", Saavedra 15 3 piso, (C1083ACA) Buenos Aires, Argentina}
\email{cconde@ungs.edu.ar}

\address{$^4$ School of Science,
 Chongqing University of Posts and
Telecommunications, Chongqing 400065, China.}
\email{zhpp04010248@126.com}



\begin{abstract}
We extend an operator P\'{o}lya--Szeg\"{o} type inequality involving the operator geometric mean to any arbitrary operator mean under some mild conditions. Utilizing the Mond--Pe\v{c}ari\'c method, we present some other related operator inequalities as well.
\end{abstract}

\maketitle
\section{Introduction}

Let ${\mathbb B}({\mathcal H})$ denote the $C^*$-algebra of all bounded linear operators on a complex Hilbert space $({\mathcal H},\langle \cdot, \cdot \rangle)$ equipped with the operator norm $\|\cdot\|$. Throughout the paper, a capital letter means an operator in ${\mathbb B}({\mathcal H})$. We identify a scalar with the identity operator $I$ multiplied by this scalar. An operator $A$ is called positive if $\langle Ax,x\rangle\geq0$ for all $x\in{\mathcal H }$, and we then write $A\geq 0$. We denote by  $A>0$  if it is a positive invertible operator. For self-adjoint operators $A, B\in{\mathbb B}({\mathcal H})$, we say $B\geq A$ if $B-A \geq 0$. A linear map $\Phi:{\mathbb B}({\mathcal H})\to{\mathbb B}({\mathcal K})$ is called positive if $A \geq 0$ implies $\Phi(A)\geq 0$. It is said to be unital if $\Phi$ preserves the identity operator.

The axiomatic theory for operator means of positive invertible operators have been developed by Kubo and Ando \cite{ando}, in particular an operator mean $\sigma$ has the monotonicity property if $A\leq C$ and $B\leq D $ imply
$A\sigma B\leq C\sigma D$. A continuous real valued
function $h$ defined on an interval $J$ is called operator monotone if $A\geq B$ implies that $h(A)\geq h(B)$ for all self-adjoint operators $A, B $ with spectra in $J$. There exists an affine order isomorphism between the class of operator means $\sigma$ and the class of positive operator monotone functions $h$ defined on $(0,\infty)$ with $h(1)=1$ via
$h(t)I=I\sigma(tI)\,\,(t>0)$. In addition, $A\sigma
B=A^{1\over 2}h(A^{-1\over2}BA^{-1\over2})A^{1\over2}$ for all strictly positive operators $A, B$. The operator monotone function
$h$ is called the representing function of $\sigma$. Using a standard limit argument, this notion can be extended for positive operators $A, B$. The operator mean corresponding to the operator monotone functions $h(t)=t^{\epsilon}\,\,(0<\epsilon<1)$ is called the  weighted operator geometric mean, which is indeed $A\sharp_\epsilon B = A^{1/2}(A^{-1/2}BA^{-1/2})^{\epsilon}A^{1/2}$. The case where $\epsilon=1/2$ gives rise to the usual operator geometric mean $\sharp$.

Ando \cite{AND} proved the property $\Phi(A\sharp B) \leq \Phi(A)\sharp\Phi(B)$ for any positive linear map $\Phi$ (not neccessarily unital). Further, it is well-known that Ando's proof can work not only for $\#$ but also for every operator mean $\sigma$. As a complementary to Ando's inequality, Seo \cite{SEO} obtained some additive and multiplicative types of reverses of Ando's inequality. Reverses of this inequality are known as operator P\'{o}lya--Szeg\"{o} type inequalities in the literature. Moslehian et al. \cite[Theorem 2.1]{BZ02} presented an operator P\'{o}lya--Szeg\"{o} inequality (see also \cite{LEE} for a proof for matrices) as follows:
\begin{theorem}\label{thm1}
Let $\Phi$ be a positive linear map. If $0<m \leq A, B\leq M$ for some positive real numbers $m\leq M$, then
 \begin{eqnarray}\label{amm}
\Phi(A)\sharp\Phi(B) \leq \frac{M+m}{2\sqrt{Mm}} \Phi(A\sharp B)\,.
\end{eqnarray}
\end{theorem}
It is well known that $t^2$ is not operator monotone. However, Fujii et al. \cite[Theorem 6]{FUJ} applied the Kantrovich inequality to show that $t^2$ is order preserving in a certain sense. Lin \cite{LIN1} reduced the study of squared operator inequalities to that of some norm inequalities. Hoa, Toan and Binh \cite{HoaToan} studied some reverse inequalities for arbitrary pair of operator means. Recently, Fu and Hoa \cite{HoaFu} extended this result for any power greater than $1$, see also \cite{MF}.

This note intends to extend the operator P\'{o}lya--Szeg\"{o} inequality (\ref{amm}) for any arbitrary operator mean. Our results provide some general reverses of the Ando inequality under some mild conditions.


\section{Results}

In \cite{Ando-Hiai}, Hiai and Ando showed that if the following inequality
\begin{equation}\label{AAHH}
f(A\nabla B) \ge f(A\sigma B)
\end{equation}
holds for any pair of positive definite matrices $A, B$, and for the arithmetic mean $\nabla$ and for some symmetric operator mean $\sigma$ ($\neq \nabla$), then $f$ is operator monotone. Without the condition of operator monotonicity and under some mild conditions, we have the following result of the type of \eqref{AAHH}.
\begin{theorem}\label{thm1}
Let $\Phi, \Psi$ be unital positive linear maps, $\sigma, \tau$ be operator means, $0<m \leq A, B\leq M$ for some positive real numbers $m\leq M$ and $f, g: [0,\infty) \to [0,\infty)$ monotone increasing functions. Then
\begin{eqnarray*}
f\left(\Phi(A)\sigma \Phi(B)\right) \leq f(M)g(m^{-1}) g(\Psi(A) \tau \Psi(B)),
\end{eqnarray*}
\begin{eqnarray*}
f\left(\Phi(A)\sigma \Phi(B)\right) \leq f(M)g(m^{-1}) g(\Psi(A \tau B)),
\end{eqnarray*}
\begin{eqnarray*}
f\left(\Phi(A\sigma B)\right) \leq f(M)g(m^{-1}) g(\Psi(A) \tau \Psi(B)),
\end{eqnarray*}
and
\begin{eqnarray*}
f\left(\Phi(A\sigma B)\right) \leq f(M)g(m^{-1}) g(\Psi(A \tau B)).
\end{eqnarray*}
\end{theorem}
\begin{proof}
Put
\[
C := \Phi(A\sigma B) \,\,\textrm{or}\,\, \Phi(A)\sigma \Phi(B) ,\,\, D := \Psi(A)\tau\Psi(B)\, \,\textrm{or}\,\, \Psi(A\tau B).
\]
Since $m\leq A, B\leq M$, we have
$$m=m(1\sigma 1)\leq m\sigma m\leq A\sigma B\leq M\sigma M\leq M(1\sigma 1)=M.$$
Therefore, $ m\leq \Phi(A\sigma B)\leq M$. Similarly, we also have $$m\leq\Phi(A)\sigma \Phi(B),  \Psi(A)\tau\Psi(B), \Psi(A\tau B)\leq M,$$
i.e., $$m\leq C, D\leq M.$$
It follows that
\begin{equation*}
\begin{split}
\|g^{\frac{-1}{2}}(D)f^{\frac{1}{2}}(C)\| & \le (||f(C)|| ||g(D)^{-1}||)^{\frac{1}{2}} \\
& \le  (f(M)g(m^{-1}))^{\frac{1}{2}},
\end{split}
\end{equation*}
which leads to the desired results.
\end{proof}
\begin{remark}
The  positive linear maps $\Phi$ and $\Psi$ in Theorem \ref{thm1} must be unital. If not, the conclusion doesn't hold in general. For example, Let $f(x)=g(x)=x$, $\Phi(A)=kA$, $\Psi(A)=lA$, where $k,l>0$. Obviously, the conclusion doesn't hold in Theorem \ref{thm1} when we choose the appropriate positive numbers  $k$ and $l$.

Notice that if $g$ is an operator monotone function on $[0, \infty),$ then from the general Ando inequality $\Psi(A) \tau \Psi(B) \ge \Psi(A \tau B)$ we have
$$
g(\Psi(A) \tau \Psi(B)) \ge g(\Psi(A \tau B)).
$$
Without the condition of operator monotonicity on $g$ the last inequality can be false.
\end{remark}

 When $f: [0,\infty) \to [0,\infty)$ is a monotone increasing function, we know that $f^{p}$ is also a monotone increasing function for $p>0$. Then we have the following corollary.
\begin{corollary}\label{coro1}
Let $\Phi, \Psi$ be positive linear maps, $\sigma, \tau$ be operator means, $0<m \leq A, B\leq M$ for some positive real numbers $m\leq M$ and $f, g: [0,\infty) \to [0,\infty)$ monotone increasing functions. Then
\begin{eqnarray*}
f^{p}\left(\Phi(A)\sigma \Phi(B)\right) \leq f^{p}(M)g^{p}(m^{-1}) g^{p}(\Psi(A) \tau \Psi(B)),
\end{eqnarray*}
\begin{eqnarray*}
f^{p}\left(\Phi(A)\sigma \Phi(B)\right) \leq f^{p}(M)g^{p}(m^{-1}) g^{p}(\Psi(A \tau B)),
\end{eqnarray*}
\begin{eqnarray*}
f^{p}\left(\Phi(A\sigma B)\right) \leq f^{p}(M)g^{p}(m^{-1}) g^{p}(\Psi(A) \tau \Psi(B)),
\end{eqnarray*}
and
\begin{eqnarray*}
f^{p}\left(\Phi(A\sigma B)\right) \leq f^{p}(M)g^{p}(m^{-1}) g^{p}(\Psi(A \tau B)),
\end{eqnarray*}
where, $p>0$.
\end{corollary}

\begin{remark}
It is well-known that the inequality
$$
A\nabla B \ge A \sigma B
$$
could not be squared, and the function $f(x)= x^2$ is not operator monotone on $[0, \infty)$. From Theorem \ref{thm1}, for any monotone increasing function $f$ on $(0, \infty)$ we have
\begin{eqnarray*}
f(M)f(m^{-1}) f(A \sigma B) \ge  f(A \nabla B).
\end{eqnarray*}
Notice that the value $f(M)f(m^{-1})$ is greater than the Kantorovich constant of $M > m$, but
$$
K(M, m) (A\sharp B)^2 \ge (A\nabla B)^2.
$$ is not true in general.
A counterexample can be found in \cite{HoaToan}. But for reader's convenience, we give here an example.
Indeed, let's take $m= 1, M= 2$ and the following matrices
\begin{equation*}\label{giatriXY}
X = \left(\begin{array}{cc}
0.0688 & -0.1082\\
-0.1082 & 0.1998 \\
 \end{array}
\right),\
Y = \left(\begin{array}{cc}
0.7489 & 0.1237\\
0.1237 & 0.4212 \\
\end{array}
\right).
\end{equation*}
It is obvious that  $m \le  X, Y \le M.$ With a help of Matlab we get
$$
\det (K(h) (X\sharp Y)^2 - (X \nabla Y)^2) = -0.0014,
$$
Hence, the inequality
\begin{equation}\label{Q}
K(h) (X\sharp Y)^2 \ge (X \nabla Y)^2
\end{equation}
does not hold.
\end{remark}

It is clear that if we replace the means $\sigma$ and $\tau$ in Theorem \ref{thm1} by the arithmetic mean and the geometric mean, respectively, then we can get the following inequality for $n$ positive definite matrices $A_i$:
\begin{equation}
f\left(\Phi\left(\frac{1}{n}\sum_{i=1}^n A_i \right)\right) \leq f(M)g(m^{-1}) g(\Psi(G(A_1, A_2,\cdots, A_n)),
\end{equation}
where $G(A_1, A_2,\cdots, A_n)$ is the ALM geometric mean of matrices $\{A_n\}$. Recall that for the geometric means for several variables, there are two different approaches: one of them is the iteration approach due to Ando, Li and Mathias \cite{ALM}, which is called the ALM geometric mean, and the other is the Riemannian geometry approach due to Moakher \cite{MOA}, and Bhatia and Holbrook \cite{BH}.  It can easily be verified by the iteration argument from the two variable case that for a positive map $\Psi$
$$
\Psi(G(A_1, A_2,\cdots, A_n)) \le G(\Psi(A_1), \Psi(A_2),\cdots, \Psi(A_n)),
$$
whence we get the next result.
\begin{corollary}\label{MultiAGI}
Let $\Phi, \Psi$ be positive linear maps, $\sigma, \tau$ arbitrary operator means, $0<m \leq A_1, A_2, \cdots, A_n \leq M$ for some positive real numbers $m\leq M$. Let $f, g: [0,\infty) \to [0,\infty)$ be monotone increasing functions and let $g$ be operator monotone on $[0,\infty)$. Then
 \begin{eqnarray*}
f\left(\Phi\left(\frac{1}{n}\sum_{i=1}^n A_n\right)\right) \leq f(M)g(m^{-1}) g(G(\Psi(A_1), \Psi(A_2),\cdots, \Psi(A_n)))\,.
\end{eqnarray*}
\end{corollary}
\begin{remark}
In the special case when $f(x)= g(x)=x$, $\Psi (A)= \Phi(A) = A$ we get a reverse arithmetic-geometric mean (AGM) inequality for multi-geometric and multi-arithmetic means as follows:
\begin{eqnarray*}
\frac{1}{n}\sum_{i=1}^n A_n \leq \frac{M}{m}G(A_1, A_2,\cdots, A_n),
\end{eqnarray*}
Recall the AGM inequality obtained by Yamazaki \cite{Yam}
$$
\frac{A_1+A_2+ \cdots + A_n}{n} \le \left(\frac{(m+M)^2}{4Mm}\right)^{\frac{n-1}{2}}G(A_1, A_2, \cdots, A_n).
$$
Let us compare the Kantorovich constant $(\frac{(m+M)^2}{4Mm})^{\frac{n-1}{2}}$ and $\frac{M}{m}$ in case $M=2, m=1$ and $n=5$. Then we have
$$(\frac{(m+M)^2}{4Mm})^{2} = (\frac{9}{8})^{2} \le  2 = \frac{M}{m}.$$

In general, it is easy to see that for enough big $n$, we have
$$
\left(\frac{9}{8}\right)^{(n-1)/2} \ge 2.
$$
That means that none of coefficients in Yamazaki's inequality and in Corollary \ref{MultiAGI} is uniformly better.
\end{remark}

\begin{remark}
The constants $m$ and $M$ can be chosen as $$m=\min\left(\{\langle Ax,x\rangle: \|x\|=1\}\cup \{\langle Bx,x\rangle: \|x\|=1\}\right)$$ and $$M=\max\left(\{\langle Ax,x\rangle: \|x\|=1\}\cup \{\langle Bx,x\rangle: \|x\|=1\}\right).$$
\end{remark}


Theorem \ref{thm1} for the functions $f(t)=t^p, g(t)=t^q$ with $p, q\geq 0$ gives rise to the following result.

\begin{corollary}
Let $\Phi, \Psi$ be positive linear maps, $p, q\geq 0$ and $0<m \leq A, B\leq M$ for some positive real numbers $m\leq M$. Then
\begin{eqnarray}\label{zzzz2}
\left(\Phi(A)\sharp\Phi(B)\right)^p \leq M^{p}m^{-q} (\Psi(A\sharp B))^q.
\end{eqnarray}
\end{corollary}

\begin{remark}
Although, for $p=q=1$ one observes that the coefficient $\frac{M}{m}$ in \eqref{zzzz2} is greater than $\frac{M+m}{2\sqrt{Mm}}$ in \eqref{amm} in general, we obtained the relation between $( \Phi(A)\sharp\Phi(B))^p$ and $(\Psi(A\sharp B))^q$.
\end{remark}


In the sequel, we provide another general inequality by employing the Mond--Pe\v{c}ari\'c method.

\begin{theorem}\label{thm3}
Let $\Phi$ be a positive linear map, $\sigma$ be an operator mean with the representing function $h$, $0<m \leq A, B\leq M$ for some positive real numbers $m\leq M$ and $f, g: [m,M] \to [0,\infty)$ be continuous functions such that $g$ is nonzero, monotone increasing and concave. Then
 \begin{eqnarray*}
f\left(\Phi(A)\sigma \Phi(B)\right) \leq \gamma g(\Phi(A \sigma B))\,,
\end{eqnarray*}
where \quad $\gamma=\max\left\{\frac{f(t)}{\mu_g\alpha^{-1}t+\nu_g}: m\leq t\leq M\right\}$,\quad $\mu_g:=\frac{g(M)-g(m)}{M-m}$, $\nu_g:=\frac{Mg(m)-mg(M)}{M-m}$,\quad $\alpha=\max\left\{\frac{h(t)}{\mu_h t+\nu_h}: \frac{m}{M}\leq t\leq \frac{M}{m}\right\}$,\, $\mu_h:=\frac{h(\frac{M}{m})-h(\frac{m}{M})}{(\frac{M}{m})-(\frac{m}{M})}$ and $\nu_h:=\frac{(\frac{M}{m})h(\frac{m}{M})-(\frac{m}{M})h(\frac{M}{m})}{(\frac{M}{m})-(\frac{m}{M})}$.
\end{theorem}
\begin{proof}
It follows from $ \frac{m}{M}A \leq B\leq \frac{M}{m}A$ and \cite[Corollary 5.29]{abc} that
\begin{eqnarray}\label{mond2}
\alpha \Phi(A\sigma B) \geq \Phi(A)\sigma \Phi(B),
\end{eqnarray}
where $\alpha=\max\left\{\frac{h(t)}{\mu_h t+\nu_h}: \frac{m}{M}\leq t\leq \frac{M}{m}\right\}$, $\mu_h:=\frac{h(\frac{M}{m})-h(\frac{m}{M})}{(\frac{M}{m})-(\frac{m}{M})}$ and $\nu_h:=\frac{(\frac{M}{m})h(\frac{m}{M})-(\frac{m}{M})h(\frac{M}{m})}{(\frac{M}{m})-(\frac{m}{M})}$. Since $g$ is a concave function, $g(t)\geq \mu_g t+\nu_g$ for all $t\in[m,M]$, where  $$\mu_g:=\frac{g(M)-g(m)}{M-m}\qquad {\rm and} \qquad \nu_g:=\frac{Mg(m)-mg(M)}{M-m}.$$
Utilizing the continuous functional calculus and the fact that $m \leq \Phi(A\sigma B)\leq M$ we get $$g(\Phi(A\sigma B))\geq \mu_g\Phi(A\sigma B)+\nu_g.$$
It follows from \eqref{mond2} that
$$g(\Phi(A\sigma B))\geq \mu_g \alpha^{-1}\big(\Phi(A)\sigma \Phi(B)\big)+\nu_g.$$
We intend to find a scalar $\gamma$ such that $\gamma \mu_g \alpha^{-1}\big(\Phi(A)\sigma \Phi(B)\big)+\nu_g \geq f(\Phi(A)\sigma \Phi(B))$. By the functional calculus it is sufficient to find $\gamma$ in such a way that $\gamma (\mu_g\alpha^{-1}t+\nu_g) \geq f(t)$ for all $t\in [m,M]$. Thus $\gamma$ should be at least $$\max\left\{\frac{f(t)}{\mu_g\alpha^{-1}t+\nu_g}: m\leq t\leq M\right\},$$ which can be found by maximizing the one variable function $$\frac{f(t)}{\mu_g\alpha^{-1}t+\nu_g}$$ by usual calculus computations. One should note that there is no $t\geq m$ such that $\mu_g\alpha^{-1}t+\nu_g=0$.
\end{proof}
\begin{remark}
If $\sigma=\sharp_\epsilon$, then $\alpha$ in Theorem \ref{thm3} is indeed $k(\frac{m}{M}, \frac{M}{m})$ in which
\begin{align*}
k(t,s)=\frac{\epsilon^\epsilon (s-t) (st^\epsilon-ts^\epsilon)^{\epsilon-1}}{(1-\epsilon)^{\epsilon-1}(s^\epsilon-t^\epsilon)^{\epsilon}}
\end{align*}
is a Kantorovich constant.
\end{remark}

In the case when $f = g$ is a nonzero operator monotone function on $[0, \infty)$, we have the following theorem.

\begin{theorem}\label{thm4}
Let $\Phi$ be a positive linear map, $f$ a nonzero operator monotone function on $[0,\infty)$, $\tau, \sigma$ operator means between arithmetic and harmonic means, and $0 < m \le M$. Then for any positive matrices $0 < m \le A, B \le M$,
 \begin{eqnarray*}
f\left(\Phi(A)) \tau f(\Phi(B)\right) \leq K(M, m) f(\Phi(A \sigma B)).
\end{eqnarray*}
\end{theorem}

\begin{proof}
It is well-known (see, for example \cite{HoaToan}) that if $\Phi$ is unital, then
\begin{equation}\label{hoa}
K(M, m) \Phi(A \sigma B) \ge \Phi(A) \nabla \Phi(B).
\end{equation}
There is a standard argument in dealing with operator inequalities in which we may consider positive linear maps instead of unital positive linear maps. In fact, by passing to $\Phi+\varepsilon\,{\rm id}$ we can assume that $\Phi(I) >0$ and then we can define $\Psi(A)=\Phi(I)^{-1/2}\Phi(A)\Phi(I)^{-1/2}$ as a unital positive linear map. Using a limit argument, we get \eqref{hoa}.

Since $f(x)$ is operator monotone on $[0,\infty)$, hence $x/f(x)$ is operator monotone $[0,\infty)$, too. For $k\ge 1$ and for any $x\ge 0$, we have
$$
\frac{kx}{f(kx)} \ge \frac{x}{f(x)}
$$
or,
$$
f(kx) \le kf(x).
$$
Consequently, for any positive number $x$,
$$
f(K(M, m) x) \le K(M, m) f(x).
$$
From inequality (\ref{hoa}), we get
\begin{equation*}
\begin{split}
K(M, m) f(\Phi(A \sigma B)) & \ge f(K(M, m) \Phi(A \sigma B)) \\
& \ge f(\Phi(A) \nabla \Phi(B)) \\
& \ge f(\Phi(A)) \nabla f(\Phi(B)) \\
& \ge f(\Phi(A)) \tau f(\Phi(B)).
\end{split}
\end{equation*}
\end{proof}

\section{Concluding remark}
From Theorem \ref{thm4} for any $p\in [0, 1]$ and for any positive matrices $0 < m \le A, B \le M,$
\begin{equation}\label{q2}
A^p \nabla B^p \le K(M, m) (A \sharp B)^p.
\end{equation}
At the same time, inequality (\ref{q2}) is failed when $p=2$. Indeed, it is easy to construct some matrices $A$ and $B$ that do not satisfy (\ref{q2}). For example, with $m=0.4, M=3$ and for matrices
\begin{equation*}\label{giatriXY}
A = \left(\begin{array}{cc}
1.3096 & 0.4414\\
0.4414 & 0.6204 \\
 \end{array}
\right),\quad
B = \left(\begin{array}{cc}
0.7062 & 1.1641\\
1.1641 & 2.1050 \\
\end{array}
\right),
\end{equation*}
we have
$$det(K(M, m) (A \sharp B)^2 - A^2 \nabla B^2)=-0.4111$$

It was shown in \cite{Ando-Hiai} that the inequalities
$$
f(A\sharp B) \le f(A\nabla B) \quad \hbox{and} \quad f(A\nabla B) \le f(A) \sharp f(B)
$$
characterize operator monotone (decreasing, respectively) functions on $[0, \infty)$. It was  also proved in \cite{hoaLAA} that an additive reverse inequality
$$
f(A\nabla B) \le f(A\sharp B+ \frac{1}{2}A^{1/2}|I-A^{-1/2}BA^{-1/2}|A^{1/2})
$$
characterizes operator monotonicity.

So, it is natural to ask the following question: {\em Does the multiplicative reverse inequality (\ref{hoa}) characterize operator monotone functions? In particular, suppose that
\begin{equation*}\label{hoa}
 f(A) \nabla f(B) \le K(M, m) f(A \sharp B)
\end{equation*}
holds for any positive matrices $0 < m \le A, B \le M$ (where $m < M$ are the given positive numbers). Is it true that the function $f$ is operator monotone on $[m, M]$?}
\bigskip

\textbf{Acknowledgements} M. S. Moslehian (the corresponding author) would like to thank The Abdus Salam International Centre for Theoretical Physics (ICTP) for the opportunity provided to develop the idea of paper during his stay in Trieste, Italy.

\bibliographystyle{amsplain}

\end{document}